\documentclass[12pt]{amsart}
\usepackage{amssymb}
\newtheorem{thm}{Theorem}

\newcommand{\intprod}{\,\rule{5pt}{.3pt}\rule{.3pt}{7pt}\;}

\newcommand{\xoo}[3]{\begin{picture}(40,11)
\put(4,1.5){\makebox(0,0){$\times$}}
\put(20,1.5){\makebox(0,0){$\bullet$}}
\put(36,1.5){\makebox(0,0){$\bullet$}}
\put(4,1.5){\line(1,0){32}}
\put(4,8){\makebox(0,0){$\scriptstyle #1$}}
\put(20,8){\makebox(0,0){$\scriptstyle #2$}}
\put(36,8){\makebox(0,0){$\scriptstyle #3$}}
\end{picture}}
\newcommand{\ooo}[3]{\begin{picture}(40,11)
\put(4,1.5){\makebox(0,0){$\bullet$}}
\put(20,1.5){\makebox(0,0){$\bullet$}}
\put(36,1.5){\makebox(0,0){$\bullet$}}
\put(4,1.5){\line(1,0){32}}
\put(4,8){\makebox(0,0){$\scriptstyle #1$}}
\put(20,8){\makebox(0,0){$\scriptstyle #2$}}
\put(36,8){\makebox(0,0){$\scriptstyle #3$}}
\end{picture}}
\newcommand{\xox}[3]{\begin{picture}(40,11)
\put(4,1.5){\makebox(0,0){$\times$}}
\put(20,1.5){\makebox(0,0){$\bullet$}}
\put(36,1.5){\makebox(0,0){$\times$}}
\put(4,1.5){\line(1,0){32}}
\put(4,8){\makebox(0,0){$\scriptstyle #1$}}
\put(20,8){\makebox(0,0){$\scriptstyle #2$}}
\put(36,8){\makebox(0,0){$\scriptstyle #3$}}
\end{picture}}
\newcommand{\btwo}[2]{\begin{picture}(28,11)
\put(4,1.5){\makebox(0,0){$\times$}}
\put(24,1.5){\makebox(0,0){$\bullet$}}
\put(5,2.5){\line(1,0){20}}
\put(5,0.5){\line(1,0){20}}
\put(14,1.5){\makebox(0,0){$\rangle$}}
\put(4,8){\makebox(0,0){$\scriptstyle #1$}}
\put(24,8){\makebox(0,0){$\scriptstyle #2$}}
\end{picture}}

\newcommand{\gtwo}[2]{\begin{picture}(28,10)
\put(4,1.1){\makebox(0,0){$\bullet$}}
\put(24,1.5){\makebox(0,0){$\times$}}
\put(4,3.5){\line(1,0){18}}
\put(4,1.5){\line(1,0){20}}
\put(4,-.5){\line(1,0){18}}
\put(14,1.5){\makebox(0,0){$\langle$}}
\put(4,8){\makebox(0,0){$\scriptstyle #1$}}
\put(24,8){\makebox(0,0){$\scriptstyle #2$}}
\end{picture}}
\DeclareSymbolFont{script}{U}{eus}{m}{n}
\DeclareMathSymbol{\Wedge}{0}{script}{"5E}
\begin{document}
\title[Flying saucers]{Aerodynamics of flying saucers}
\author[Michael Eastwood]{Michael Eastwood}
\address{\hskip-\parindent
School of Mathematical Sciences\\
University of Adelaide\\ 
SA 5005\\ 
Australia}
\email{meastwoo@gmail.com}
\author[Pawe\l\ Nurowski]{Pawe\l\ Nurowski}
\address{\hskip-\parindent
Centrum Fizyki Teoretycznej,
Polska Akademia Nauk, Al.~Lotnik\'ow 32/46, 02-668 Warszawa, Poland}
\email{nurowski@cft.edu.pl}
\subjclass{53A20, 53A30, 53A40}
\thanks{This work was supported by the Simons Foundation grant 346300 and the
\mbox{Polish} Government MNiSW 2015--2019 matching fund. It was written whilst
the first \mbox{author} was visiting the Banach Centre at IMPAN in Warsaw for
the Simons \mbox{Semester} `Symmetry and Geometric Structures' and during
another visit to \mbox{Warsaw} supported by the Polish National Science Centre
(NCN) via the POLONEZ grant 2016/23/P/ST1/04148, which received funding from
the European Union's Horizon 2020 research and innovation programme under the
Marie Sk{\l}odowska-Curie grant agreement No.~665778.}

\begin{abstract} We identify various structures on the configuration space $C$
of a flying saucer, moving in a three-dimensional smooth manifold~$M$. Always
$C$ is a five-dimensional contact manifold. If $M$ has a projective structure,
then $C$ is its twistor space and is equipped with an almost contact Legendrean
structure. Instead, if $M$ has a conformal structure, then the saucer moves
according to a CR structure on~$C$. With yet another structure on $M$, the
contact distribution in $C$ is equipped with a cone over a twisted cubic. This
defines a certain type of Cartan geometry on~$C$ (more specifically, a type of
`parabolic geometry') and we provide examples when this geometry is `flat,'
meaning that its symmetries comprise the split form of the exceptional Lie
algebra~${\mathfrak{g}}_2$.
\end{abstract}
\renewcommand{\subjclassname}{\textup{2010} Mathematics Subject Classification}
\maketitle

\setcounter{section}{-1}\vspace{-2pt}
\section{Introduction}
Throughout this article $M$ will be a $3$-dimensional smooth oriented manifold.
For $x\in M$, a non-zero element $\omega\in T_x^*M$ defines an oriented 
$2$-plane $\{X\in T_xM\mid X\intprod\omega=0\}$ at~$x$. 
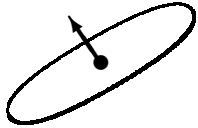
\begin{center}\setlength{\unitlength}{2pt}
\begin{picture}(40,25)(0,33)
\linethickness{.7pt}
\qbezier (37.32,55) (34.82,59.33) (17.50,49.33)
\qbezier (17.50,49.33) (0.18,39.33) (2.68,35)
\linethickness{1pt}
\qbezier (2.68,35) (5.18,30.67) (22.5,40.67)
\qbezier (22.5,40.67) (39.82,50.67) (37.32,55)
\put(20,45){\makebox(0,0){\large$\bullet$}}
\thicklines
\put(20,45){\vector(-2,3){6}}
\put(20.15,45.1){\line(-2,3){5}}
\put(19.85,44.9){\line(-2,3){5}}
\end{picture}
\end{center}
Thus, we may realise the bundle of oriented two-planes in $TM$ as
\begin{equation}\label{C}
\frac{\{\omega\in T^*M\setminus\{\mbox{the zero section}\}\}}
{\omega\sim\lambda\omega\enskip\mbox{for}\enskip\lambda>0}
={\mathrm{Gr}}_2^+(TM)\xrightarrow{\,\pi\,}M.\end{equation}
We shall write $C\xrightarrow{\,\pi\,}M$ for this {\em configuration space\/}
of oriented saucers in~$M$ (flying saucers are oriented as they traditionally
have a cockpit). 
The basic intrinsic structure on $C$ is a contact distribution 
$H\subset TC$. We shall see that a saucer moves along a path in $C$ that is
everywhere tangent to $H$ if and only if its motion in $M$ is in directions
taken from its own disc. With this specification, arbitrary `rolls' are
allowed but, with more structure on~$M$, these rolls are constrained and
this gives rise to differential geometries on $C$ expressed in terms of various
algebraic structures on~$H$.

This article is organised as follows. Section~\ref{contact} discusses the
contact geometry on~$C$. In Section~\ref{attack} it is supposed that $M$ has a
projective structure and consequently we shall find an almost contact
Legendrean structure on~$C$ (originally due to Takeuchi~\cite{T} using methods
due to Tanaka). Then $M$ is projectively flat if and only if $C$ has maximal
symmetry ${\mathfrak{sl}}(4,{\mathbb{R}})$. Usual CR structures emerge in
Section~\ref{landing} (following LeBrun~\cite{LeB1}). In each case there are
links with twistor theory. Section~\ref{G2} explains how to endow $C$ with a
geometric structure modelled on the contact homogeneous space for the split
form of the exceptional Lie group~$G_2$ and we present some examples for which
this structure turns out to be `flat,' meaning that it is locally isomorphic to
the flat model.
Section~\ref{geometry_on_M} further investigates the geometry on $M$ that is 
needed to construct this `$G_2$ contact structure' on~$C$.

This article is concerned only with the geometry of $C$ and especially its
construction from, and relationship to, various geometrical features on~$M$. In
a companion article~\cite{EN}, we simply start with Euclidean space
$M={\mathbb{R}}^3$ and explain how the various aerodynamic options considered
here are reflected in the aerobatic man{\oe}uvres available to a pilot flying
according to these options.

The authors would like to thank Katja Sagerschnig and Travis Willse for many
helpful conversations.

\section{The contact structure on $C$}\label{contact}
In fact, the intrinsic contact structure on $C$ is, in addition, filtered. 
Specifically, we shall find canonically defined subbundles
$$TC\supset H\supset V,$$
where $H$ is the contact distribution and $V$ is 
the vertical subbundle of $\pi:C\to M$.
To define~$H$, we note the
canonical identification
$$TC/V=\pi^*TM$$
and observe that a point in $C$ is precisely a point $x\in M$ together with an
oriented $2$-plane in~$T_xM$. In other words, we have a tautologically defined
rank $2$ subbundle $P\subset\pi^*TM$ recording this subspace and we may
define $H$ as the inverse image of $P$ under $TC\to\pi^*TM$. In summary,
we have a canonical filtration  
\begin{equation}\label{canonical_filtration}
TC=\overbrace{\vphantom{(}L+{}\hspace{10pt}}^{\makebox[0pt]{$=\pi^*TM$}}
\hspace{-11pt}\underbrace{P+V\!}_{\textstyle= H}\,,\end{equation}
where $L$ is, by definition, the line bundle $TC/H$ and we are recording here 
the composition factors, with the rightmost bundle $V$ being the 
natural subbundle. 

It remains to see that $H$ is, indeed, a contact distribution. This is a
calculation in local co\"ordinates. Specifically, we recall that the cotangent
bundle $T^*M$ of any smooth manifold is equipped with the well-known
tautological $1$-form~$\theta$. In `canonical co\"ordinates' $(x^a,p_a)$
on~$T^*M$, we have $\theta=p_a\,dx^a$ (for details, see~\cite{C}). On $C$, we 
may use an affine chart $(x,y,z,a,b)\mapsto(x,y,z,a,b,1)$ to embed $C$ in 
$T^*M$ and pull-back $\theta$ to the $1$-form $a\,dx+b\,dy+dz$ whose kernel 
is~$H$. Then 
$$d(a\,dx+b\,dy+dz)=da\wedge dx+db\wedge dy$$
is the Levi form on~$H$, which is manifestly non-degenerate.

Alternatively, the Levi form on $H$ may be seen as arising from the canonical
symplectic form on $T^*M$ as follows. In canonical co\"ordinates $(x^a,p_a)$
on~$T^*M$, the symplectic form is $d\theta=dp_a\wedge dx^a$. It means that if
we use canonical co\"ordinates on the total space of
$T^*M\xrightarrow{\,\nu\,}M$ to split its tangent bundle as
\begin{equation}\label{splitting}T(T^*M)=
\begin{array}c\nu^*TM\\[-4pt] \oplus\\[-2pt] \nu^*T^*M\end{array}\ni
\left[\!\!\begin{array}{c}X^a\\ \omega_a\end{array}\!\!\right]
\leftarrowtail
\left[\!\!\begin{array}{c}0\\ p_a\end{array}\!\!\right]=\theta\end{equation}
then the symplectic form is
\begin{equation}\label{canonical_symplectic_form}
\left[\!\!\begin{array}{c}X^a\\ \omega_a\end{array}\!\!\right]\otimes
\left[\!\!\begin{array}{c}\tilde X^b\\ \tilde\omega_b\end{array}\!\!\right]
\longmapsto X^b\tilde\omega_b-\tilde X^b\omega_b,\end{equation}
(and is independent of choice of co\"ordinates (as we shall see in the next
section by a different argument)). Viewing $C$ as in~(\ref{C}), we see its
tangent bundle as
$$\left\{\left[\!\!\begin{array}{c}X^a\\ \omega_a\end{array}\!\!\right]\right\}
\raisebox{-3pt}{\LARGE$/$}\begin{array}{l}
\omega_a\sim\omega_a+tp_a\\
\mbox{for}\enskip t\in{\mathbb{R}}\end{array}$$
and $H$ as the subbundle for which $X^ap_a=0$. Evidently, the form
(\ref{canonical_symplectic_form}) descends to~$H$ (and is easily verified to be
the Levi form).

Flying tangent to $H$ in $C$ is saying exactly that the velocity of the saucer
in space is constrained to lie in its own disc. Otherwise, the pilot is free to
make arbitrary `rolls' and the Chow--Rashevskii Theorem~\cite{M} in this
context implies that a pilot flying with these man{\oe}uvres may park her craft
in an arbitrary location and orientation.

Finally in this section, we consider the abstract structure on $C$ arising
from its being a configuration space. Recall from (\ref{canonical_filtration})
that $C$ is equipped with a filtration on its tangent bundle
\begin{equation}\label{LPV}TC=L+P+V\end{equation}
in which $H=P+V$ is contact and the two-dimensional subbundle $V$ is
integrable. In fact, there are no local invariants of this arrangement.
\begin{thm} Suppose $C$ is a five-dimensional contact manifold with contact 
distribution~$H$. Suppose $V$ is a rank two integrable subbundle of~$H$. Then 
we may find local co\"ordinates $(x,y,z,a,b)$ on $C$ so that 
\begin{itemize}
\item $H$ is defined by the $1$-form\enskip $\lambda\equiv dz+a\,dx+b\,dy$,
\item $V$ is defined by $\lambda$ and the two $1$-forms\/ $dx$ and\/~$dy$. 
\end{itemize}
Thus, it is as if\/ $C$ were defined by $M$ with local co\"ordinates $(x,y,z)$.
\end{thm}
\begin{proof} The following argument pertains locally. Choose $1$-forms 
$\lambda,\mu,\nu$  so that 
$$H=\lambda^\perp\quad\mbox{and}\quad V=(\lambda,\mu,\nu)^\perp.$$
Integrability of $V$ ensures, by Frobenius, that we can find co\"ordinates 
$(x,y,z,u,v)$ so that 
$$\lambda,\mu,\nu\in{\mathrm{span}}\{dx,dy,dz\}$$
and, since $\lambda\not=0$, we may rescale it and subtract appropriate
multiples thereof from $\mu$ and $\nu$ to suppose, without loss of generality,
that
\begin{equation}\label{lambda_mu_nu}
\lambda=dz+a\,dx+b\,dy\qquad \mu=p\,dx+q\,dy\qquad\nu=r\,dx+s\,dy,
\end{equation}
for suitable smooth functions $(a,b,p,q,r,s)$. Now, since $H$ is contact,
$$0\not=\lambda\wedge d\lambda\wedge d\lambda
=2\,dx\wedge dy\wedge dz\wedge da\wedge db$$
so $(x,y,z,a,b)$ may be used as local co\"ordinates instead. Finally, 
$$0\not=\lambda\wedge\mu\wedge\nu=(ps-qr)\,dx\wedge dy\wedge dz$$
so we may replace $\{\mu,\nu\}$ by $\{dx,dy\}$ without changing 
their span. \end{proof}

\section{The almost contact Legendrean geometry on $C$}\label{attack}

Firstly, we revisit the splitting (\ref{splitting}), now using torsion-free
connections instead of choosing co\"ordinates. As is well-known~\cite{L}, a
connection on $T^*M$ may be viewed as a splitting~(\ref{splitting}) of 
$T(T^*M)$, into 
{\em horizontal\/} and vertical subbundles. If we change connections, say
\begin{equation}\label{connection_change}
\widehat\nabla_aX^c=\nabla_aX^c+\Gamma_{ab}{}^cX^b,\end{equation}
(using Penrose's abstract index
notation~\cite{OT}) then the splitting changes according to
\begin{equation}\label{change_of_splitting}
\widehat{\left[\!\!\begin{array}{c}X^b\\ \omega_b\end{array}\!\!\right]}
=\left[\!\!\begin{array}{c}X^b\\ 
\omega_b+X^a\Gamma_{ab}{}^cp_c\end{array}\!\!\right].\end{equation}
Now, if we insist on using {\em torsion-free\/} connections, as we may, then
the {\em skew\/} form (\ref{canonical_symplectic_form}) is manifestly invariant
because $\Gamma_{ab}{}^c$ is {\em symmetric}.

In order to navigate in~$M$, we now suppose that this manifold is endowed with
a {\em projective differential geometric\/} structure. A detailed discussion,
specifically in $3$ dimensions, may be found in~\cite{DE}. We shall therefore
be brief in recalling the salient features. Although a projective structure may
be viewed as a type of path geometry (eminently suitable for flying in~$M$) an
operational viewpoint on projective structures is as an equivalence class of
torsion-free connections, where the notion of equivalence is that
\begin{equation}\label{projective_change}
\widehat\nabla_a\phi_b=\nabla_a\phi_b-\Upsilon_a\phi_b-\Upsilon_b\phi_a
\end{equation}
for an arbitrary $1$-form~$\Upsilon_a$. In (\ref{connection_change}) it means
that
$$\Gamma_{ab}{}^c=\Upsilon_a\delta_b{}^c+\Upsilon_b\delta_a{}^c$$
where $\delta_b{}^c$ is canonical pairing between vectors and covectors. Hence,
with (\ref{projective_change}) in place, 
the formula (\ref{change_of_splitting}) for the change in splitting becomes
\begin{equation}\label{projective_change_of_splitting}
\widehat{\left[\!\!\begin{array}{c}X^b\\ \omega_b\end{array}\!\!\right]}
=\left[\!\!\begin{array}{c}X^b\\ 
\omega_b+X^a\Upsilon_ap_b+X^cp_c\Upsilon_b\end{array}\!\!\right].\end{equation}
But with a chosen connection and hence a chosen splitting in place,
$$P\enskip\mbox{is the subspace of}\enskip\pi^*TM\enskip
\mbox{given by}\enskip\{X^b\mid X^bp_b=0\}$$
and, in any case $V$ is the quotient of $\pi^*T^*M$ given by 
$$\{\omega_b\}/\sim\enskip\mbox{where}\enskip
\omega_a\sim\omega_b+tp_a\enskip\forall\,t\in{\mathbb{R}}.$$
{From} (\ref{projective_change_of_splitting}) it follows at once that the
splitting $H=P\oplus V$ is projectively invariant. In summary, we
have proved the following. 
\begin{thm}\label{projective_theorem}
A projective structure on~$M$ gives rise to extra structure on its
configuration space~$C$. Specifically, the contact distribution $H$ canonically
splits as
\begin{equation}\label{H_splits}H=P\oplus V.\end{equation}
Both $P$ and $V$ are null with respect to the Levi form
${\mathcal{L}}:\Wedge^2H\to L$, which otherwise restricts to a non-degenerate
pairing $P\otimes V\to L$. 
\end{thm}

In general, if $C$ is a manifold with contact distribution $H\subset TC$, then
a splitting $H=P\oplus V$ into null subspaces for the Levi form is
a type of parabolic geometry~\cite{CS} called {\em almost contact Legendrean}.
Projective differential geometry is another type of parabolic geometry and the
construction of this section may be viewed in Dynkin diagram notation 
as $\xox{}{}{}\xrightarrow{\,\pi\,}\xoo{}{}{}$. Furthermore,
\begin{equation}\label{dynkin}TC=L+
\begin{array}{c}P\\[-4pt]
\oplus\\[-1pt]
V\end{array}=
\;\xox{1}{0}{1}+
\begin{array}{c}\xox{1}{1}{-1}\\[-2pt]
\oplus\\[-1pt]
\xox{-1}{1}{1}\end{array}\end{equation}
and the harmonic curvature splits into $3$ pieces
$$H^2({\mathfrak{g}}_{-1},{\mathfrak{sl}}(4,{\mathbb{R}}))
=\begin{array}{l}\xox{-4}{1}{2}
\;\leftrightsquigarrow
\Big\{\!\!\begin{tabular}{l}obstruction to\\
integrability of $P$\end{tabular}\\[-9pt]
\quad\;\oplus\\[-1pt]
\xox{-3}{4}{-3}\\[-2pt]
\quad\;\oplus\\[-6pt]
\xox{2}{1}{-4}\;\leftrightsquigarrow
\Big\{\!\!\begin{tabular}{l}obstruction to\\
integrability of $V$.\end{tabular}
\end{array}$$
Meanwhile, as detailed in~\cite{DE}, the harmonic curvature of $3$-dimensional
projective geometry (usually known as the projective {\em Weyl curvature}) lies
in $\xoo{-4}{1}{2}$. One can easily check that, in case $C$ is constructed from
such a $3$-dimensional projective~$M$, as above, then the harmonic curvature of
$C$ lies only in $\xox{-4}{1}{2}$ and that it is the pull-back of the Weyl 
curvature. In particular, the contact Legendrean structure on $C$ is flat if 
and only if the projective structure on $M$ is flat.

Aerobatics may now be restricted by requiring, not only that the trajectory in
$C$ be everywhere tangent to~$H$, but also that the tangent vector be null with
respect to the neutral signature conformal metric on $H$ given by the
non-degenerate pairing $P\otimes V\to L$. Some special man{\oe}uvres are
permitted. Firstly, there is the option of remaining stationary in $M$ whilst
changing the saucer orientation arbitrarily. In other words, since the fibres
of $\pi:C\to M$ are null, it is permitted to move along them as one wishes. The
second option is to move along a projective geodesic in~$M$, with any initial
orientation, lifting this curve to $C$ in accordance with the projectively
invariant splitting (\ref{H_splits}). It is a common experience in usual
aerobatics, that one carries along one's own frame of reference! Indeed, any
curve starting at $x\in M$ with an initial choice of orientation in
$\pi^{-1}(x)$ can be uniquely lifted into $C$ in accordance
with~(\ref{H_splits}). This may be viewed as the difference between `gliding'
and `powered flight.' In any case, null man{\oe}uvring now has the geometric
interpretation that, when moving in~$M$, `rolls' are restricted to be about
one's axis of flight (the `slow roll' in usual aerobatics). Using only the two
special man{\oe}uvres of stationary rolling and gliding, as above, it is
already clear that a pilot may park her craft in an arbitrary location and
orientation.

\section{A CR structure on $C$}\label{landing}
In the previous section we saw that a projective structure on $M$ is exactly
what is needed to define what might be called `attack mode,' in which a saucer
is permitted only to make rolls about its axis of flight. In coming in to land,
however, this type of man{\oe}uvre is unsuitable, even dangerous! More suitable
for landing is the motion often observed in falling leaves, whereby rolls are
constrained to be about axes orthogonal to the direction of flight. To make
sense to this `landing mode,' one clearly needs a notion of orthogonality in
the disc of the saucer. It is natural to suppose that this notion is induced
from $M$ itself. In other words, we shall suppose that $M$ is endowed with a
{\em conformal metric}.

If two Riemannian metrics $g_{ab}$ and $\widehat{g}_{ab}$ are conformally
related, it is convenient to write $\widehat{g}_{ab}=\Omega^2g_{ab}$ for a 
smooth positive function~$\Omega$. We shall suppose that $M$ is oriented and 
write $\epsilon_{abc}$ for the volume form associated to the metric~$g_{ab}$. 
A conformal change of metric 
$\widehat{g}_{ab}=\Omega^2g_{ab}$
induces a change of volume form 
$\widehat{\epsilon}_{abc}=\Omega^3\epsilon_{abc}$ (we say that $\epsilon_{abc}$
has {\em conformal weight\/}~$3$) and the corresponding 
Levi-Civita connections are related according to
\begin{equation}\label{conformal_change_of_connection}
\widehat\nabla_a\phi_b=\nabla_a\phi_b-\Upsilon_a\phi_b-\Upsilon_b\phi_a
+\Upsilon^c\phi_cg_{ab},\end{equation} 
where $\Upsilon_a=\nabla_a\log\Omega$. We may choose a metric in the conformal
class and use its Levi-Civita connection to write
$$T(T^*M)=
\begin{array}c\nu^*TM\\[-4pt] \oplus\\[-2pt] \nu^*T^*M\end{array}\ni
\left[\!\!\begin{array}{c}X^a\\ \omega_a\end{array}\!\!\right].$$
According to (\ref{change_of_splitting}) and~
(\ref{conformal_change_of_connection}), if $\widehat{g}_{ab}=\Omega^2g_{ab}$,
then
$$\widehat{\left[\!\!\begin{array}{c}X^b\\ \omega_b\end{array}\!\!\right]}
=\left[\!\!\begin{array}{c}X^b\\ 
\omega_b+X^a\Upsilon_ap_b+X^cp_c\Upsilon_b-X_b\Upsilon^cp_c
\end{array}\!\!\right].$$
On $H$, since $\omega_b$ is only defined modulo~$p_b$, and since~$X^cp_c=0$, we
can drop two of these terms to obtain
\begin{equation}\label{conformal_change_of_splitting}
\widehat{\left[\!\!\begin{array}{c}X^b\\ \omega_b\end{array}\!\!\right]}
=\left[\!\!\begin{array}{c}X^b\\ 
\omega_b-X_b\Upsilon^cp_c
\end{array}\!\!\right]\enskip\mbox{on}\enskip
H=\begin{array}{c}P\\[-4pt] \oplus\\[-2pt] V\end{array}\!\!\!
\twoheadleftarrow\mbox{in the presence of~$g_{ab}$}.\end{equation}
Instead of $\omega_b$ up to multiples of~$p_b$, we may
use the conformal metric to suppose that $\omega^ap_a=0$ (i.e., normalise
by $\omega_b\mapsto\omega_b-(\omega^cp_c/p^ap_a)p_b$). The
change in splitting respects this normalisation (since $X^bp_b=0$). So
now we have, for a chosen metric in the conformal class,
$$H=\left\{\left[\!\!\begin{array}{c}X^b\\ \omega^b\end{array}\!\!\right]
\mbox{s.t.\ }X^bp_b=0\enskip\mbox{and}\enskip
\omega^bp_b=0\right\},$$
where $\omega^b$ has conformal weight $-2$ and, if
$\widehat{g}_{ab}=\Omega^2g_{ab}$, then (\ref{conformal_change_of_splitting})
applies. We define $J:H\to H$ by
\begin{equation}\label{J}
\left[\!\!\begin{array}{c}X^b\\ \omega^b\end{array}\!\!\right]\longmapsto
\frac{1}{\sqrt{p^dp_d}}\left[\!\!\begin{array}{c}\epsilon^{abc}X_ap_c\\
\epsilon^{abc}\omega_ap_c\end{array}\!\!\right].\end{equation}
It respects the change (\ref{conformal_change_of_splitting}) and is hence
well-defined. Since 
$$\epsilon^{abc}\epsilon_{ade}
=\delta_d{}^b\delta_e{}^c-\delta_e{}^b\delta_d{}^c$$
it follows that $J^2=-{\mathrm{Id}}$ and we have defined an 
{\em almost CR structure\/}. 
In fact, we may check that this almost CR structure is integrable as follows.
Since $P$ and $V$ are both two-dimensional,
we need only check that, for the 
Nijenhuis tensor $N(\underbar\enskip,\underbar\enskip)$, 
$$N\left(\left[\!\!\begin{array}{c}X^a\\ 0\end{array}\!\!\right],
\left[\!\!\begin{array}{c}0\\ \omega_b\end{array}\!\!\right]\right)=0.$$
This requirement expands to the vanishing of
\begin{equation}\label{what_we_would_like_to_vanish}\left[\!\!\begin{array}{c}
-\omega_b\partial^bX^a\!-\!J(\omega_b\partial^b(JX^a))
\!-\!J((J\omega_b)\partial^bX^a)\!+\!(J\omega_b)\partial^b(JX^a)\\ 
X^a\nabla_a\omega_b\!+\!J(X^a\nabla_a(J\omega_b))
\!+\!J((JX^a)\nabla_a\omega_b)\!-\!(JX^a)\nabla_a(J\omega_b)
\end{array}\!\!\right]
\end{equation}
where $\partial^a=\partial/\partial p_a$ and may be verified as follows. 
Firstly,
$$\partial^b(JX^a)
=\partial^b\Big(\frac1{\sqrt{p^ep_e}}\epsilon^{cad}X_cp_d\Big),$$
which may be expanded by the Leibniz rule as
$$\frac1{\sqrt{p^ep_e}}\epsilon^{cad}(\partial^bX_c)p_d
-\frac{p^b}{(p^ep_e)^{3/2}}\epsilon^{cad}X_cp_d
+\frac1{\sqrt{p^ep_e}}\epsilon^{cab}X_c.$$
Therefore
$$\omega_b\partial^b(JX^a)
=J(\omega_b\partial^bX^a)
+\frac1{\sqrt{p^ep_e}}\epsilon^{abc}\omega_bX_c$$
so
$$J(\omega_b\partial^b(JX^a))=-\omega_b\partial^bX^a
+\frac1{p^ep_e}(g^{ab}g^{cd}-g^{ac}g^{bd})\omega_bX_cp_d,$$
which, bearing in mind that $X^dp_d=0$ and $\omega^dp_d=0$, reduces to
$$J(\omega_b\partial^b(JX^a))=-\omega_b\partial^bX^a.$$
Similarly,
$$(J\omega_b)\partial^b(JX^a)
=J((J\omega_b)\partial^bX^a)$$
and all terms in the first line of (\ref{what_we_would_like_to_vanish}) cancel.
For the second line, it is evident that
$$X^a\nabla_a(J\omega_b)=J(X^a\nabla_a\omega_b)
\quad\mbox{and}\quad
(JX^a)\nabla_a(J\omega_b)=J((JX^a)\nabla_a\omega_b)$$
and, again, all terms cancel.

The `landing mode,' informally described at the beginning of this section is
now formally defined by the restriction $X^bJ\omega_b=0$, noting from
(\ref{conformal_change_of_splitting}) and (\ref{J}) that this constraint is
invariantly defined. Such a man{\oe}uvre is very much at odds with the `attack
mode' of the previous section. A conformal structure on $M$ does not allow slow
rolls to be defined: one sees from (\ref{conformal_change_of_splitting}) that
the restriction $X^b\omega_b=0$ is always ill-defined unless one restricts to
constant rescalings of the metric. In both modes, however, flying is restricted
by requiring that the allowed curves in $C$ are not only tangent to~$H$, but 
also that they be {\em null\/} for an appropriately defined neutral signature 
metric on~$H$ (with values in the line bundle~$L$). These metrics are
\begin{equation}\label{modes}
\left\|\left[\!\!\begin{array}{c}X\\ \omega\end{array}\!\!\right]
\right\|^2\!=\!X\intprod\omega\;\mbox{(attacking)}\qquad
\left\|\left[\!\!\begin{array}{c}X\\ \omega\end{array}\!\!\right]
\right\|^2\!=\!X\intprod J\omega\;\mbox{(landing)}.\end{equation} 
As tensors on $H$, we have the usual compatibility
$J_\alpha{}^\beta=\Omega_{\alpha\gamma}g^{\beta\gamma}$ in which any two of the
Levi form $\Omega_{\alpha\beta}$, the inverse metric $g^{\alpha\beta}$, and the
endomorphism $J_\alpha{}^\beta$ determine the third. In landing mode, we have
$J^2=-{\mathrm{Id}}$. In attacking mode, the endomorphism $J:H\to H$ instead
satisfies $J^2={\mathrm{Id}}$, being the identity on $V$ and minus the identity
on~$P$ (for any given metric in the conformal class on~$M$). For either of
these geometries, stationary rolling is allowed since $V$ is null in either
case. A metric on $M$ induces a splitting
$$0\to V\longrightarrow H\longrightarrow
\raisebox{-10pt}{\begin{picture}(0,0)
\put(-23,2){\LARGE$\smile$}
\put(-23,11){\vector(-1,1){0}}
\end{picture}}P\to 0$$
and, as in Section~\ref{attack}, the `gliding' man{\oe}uvre is now available.
In other words, we may use a horizontal lift to arrive at $\omega=0$ in either
of the neutral signature metrics~(\ref{modes}). Therefore, parking in an 
arbitrary location and orientation is easily achievable in the CR case. The 
only difference is that the splitting, and hence the particular man{\oe}uvring 
to be used, depends on choosing a metric in the conformal class.

The Dynkin diagram notation for this construction is
$$\xox{}{}{}\xrightarrow{\,\pi\,}\btwo{}{}.$$
The harmonic curvature in three-dimensional conformal geometry is the Cotton
tensor in $\btwo{-5}{4}$, which pulls back to $\xox{-3}{4}{-3}$. For more
details concerning this construction, and especially a characterisation of the
$5$-dimensional CR manifolds that arise in this way, see~\cite{LeB1} (and also
\cite{LeB2} for some very interesting consequences in the real-analytic
setting). The flat model of this construction is when 
$M=S^3\hookrightarrow S^4$ is the standard inclusion of round spheres and
$$\begin{array}{ccccl}
C&\hookrightarrow&{\mathbb{CP}}_3\\
\pi\!\!\downarrow&&\tau\!\!\downarrow\\
S^3&\hookrightarrow&S^4\end{array}$$
where $\tau:{\mathbb{CP}}_3\to S^4$ is the twistor fibration and
$C\hookrightarrow{\mathbb{CP}}_3$ is the Levi indefinite hyperquadric.

Finally, we remark that if we would like to have at our disposal both
the `attack mode' of Section~\ref{attack} and the `landing mode' of the current
section, then we require compatible projective and conformal structures on~$M$.
If $\nabla_a$ represents a projective structure and $g_{ab}$ is a metric, then
one can check that the $1$-form
$$\omega_a\equiv 4g^{bc}\nabla_bg_{ac}-\epsilon^{bcd}\nabla_a\epsilon_{bcd},$$
where $\epsilon_{abc}$ is the volume form of~$g_{ab}$, is projectively
invariant. To find a metric in the conformal class of $g_{ab}$ whose
Levi-Civita connection is in the given projective class, it is firstly
necessary that $\omega_a$ be exact. In this case, a further necessary and
sufficient condition is that
$$40\nabla_ag_{bc}+2\omega_ag_{bc}-3\omega_bg_{ac}
-3\omega_cg_{ab}=0,$$
where $\nabla_a$ has been chosen from the projective class so that
$\nabla_a\epsilon_{bcd}=0$ (cf.~\cite{MT}). In this case, we are obliged to
have a Riemannian metric on $M$ defined up to homothety, i.e.~only constant
rescalings are allowed. More severely, as discussed in~\cite{N}, the generic
projective structure does not arise from a metric at all.

In summary, for a Riemannian metric on $M$ there are two possible flying modes,
namely the `attacking mode' of \S\ref{attack}, which sees only the induced
projective structure on~$M$, and the `landing mode' of \S\ref{landing}, which
sees only the induced conformal structure on~$M$. In this case, on the
configuration space $C\to M$, there are {\em two different\/} neutral signature
conformal metrics on the contact distribution $H\subset TC$. The vertical
bundle $V\hookrightarrow H$ is null for either of these conformal metrics
on~$H$, as is the horizontal bundle $P\hookrightarrow H$ (defined only by the
projective structure on~$M$). So, when $M$ is Riemannian, a flying saucer may be
fitted with a switch that enables its pilot to toggle between attacking and
landing.

\section{A $G_2$ contact structure on $C$}\label{G2}
Recall that the contact distribution $H\subset TC$ on a five-dimensional
contact manifold $C$ is endowed with its Levi form $\Wedge^2H\to L\equiv TC/H$,
a non-degenerate symplectic form defined up to scale, thereby reducing the
structure group of frames for $H$ from ${\mathrm{GL}}(4,{\mathbb{R}})$ to the
conformal symplectic group ${\mathrm{CSp}}(4,{\mathbb{R}})$. A {\em $G_2$
contact structure\/} on $C$ is a further reduction of structure group to
${\mathrm{GL}}(2,{\mathbb{R}})\subset{\mathrm{CSp}}(4,{\mathbb{R}})$, realised
by the representation of ${\mathrm{GL}}(2,{\mathbb{R}})$ on the third symmetric
power $\bigodot^3\!{\mathbb{R}}^2$ of the standard representation of
${\mathrm{GL}}(2,{\mathbb{R}})$ on ${\mathbb{R}}^2$. As
$\bigodot^3\!{\mathbb{R}}^2$ is $4$-dimensional, we have
${\mathrm{GL}}(2,{\mathbb{R}})\hookrightarrow{\mathrm{GL}}(4,{\mathbb{R}})$
and, since
\begin{equation}\label{Wedge2H}\textstyle\Wedge^2\bigodot^3\!{\mathbb{R}}^2
=(\bigodot^4\!{\mathbb{R}}^2\otimes\Wedge^2{\mathbb{R}}^2)\oplus
(\Wedge^2{\mathbb{R}}^2\otimes\Wedge^2{\mathbb{R}}^2
\otimes\Wedge^2{\mathbb{R}}^2),\end{equation}
the second summand of which is $1$-dimensional, this homomorphism maps to
${\mathrm{CSp}}(4,{\mathbb{R}})\subset{\mathrm{GL}}(4,{\mathbb{R}})$. (See
\cite{B} for a discussion of similar reductions for the frame bundle of a
four-dimensional manifold.)

Equivalently, a $G_2$ structure on a five-dimensional contact manifold
$C$ is a rank two vector bundle $S\to C$ together with an identification
of vector bundles, compatible with the Levi form,
\begin{equation}\label{H_is_sym_3}\textstyle\bigodot^3\!S=H,\end{equation}
where $H$ is the contact distribution and $\bigodot^3\!S$ is the third
symmetric power of~$S$. It is named for a `flat model' in parabolic geometry,
namely the homogeneous space \raisebox{1pt}{$\gtwo{}{}$} for the split real
form of the Lie group~$G_2$. The tangent bundle of this homogeneous space is an
extension
\begin{equation}\label{extension}T(\raisebox{1.5pt}{$\gtwo{}{}$})=
\gtwo{0}{1}\enskip+\enskip\gtwo{3}{-1},\end{equation}
and $H$ has the required form for $S=\gtwo{1}{\enskip-1/3}$.\quad For this flat
model, the local infinitesimal symmetries are isomorphic to the split real 
form of the exceptional Lie algebra~$G_2$.

Yet a third interpretation of a $G_2$ contact structure is as a field of
twisted cubic cones inside~$H$ (akin to the interpretation of a Lorentzian
conformal structure as a field of quadratic cones, i.e.~the null vectors).
Specifically, writing $H$ in the form (\ref{H_is_sym_3}) defines a cone
$$\textstyle s\odot s\odot s\in\bigodot^3\!S=H$$
of simple vectors in each fibre of~$H$. It is easy to check that this cone
determines the structure. Geometrically, it can be regarded as a smoothly
varying family 
of twisted cubics in the bundle ${\mathbb{P}}(H)$ of three-dimensional
projective spaces (again, see~\cite{B} for the corresponding geometry in four
dimensions).

In any case, our aim in this section is to equip the configuration space of a
flying saucer with a $G_2$ contact structure.

In Sections~\ref{attack} and~\ref{landing}, the structures on the configuration
space $C$ were nicely determined by suitable differential geometric structures
on~$M$, specifically {\em projective\/} in Section~\ref{attack} and {\em
conformal\/} in Section~\ref{landing}. Now we shall firstly suppose that $M$
has a projective structure and also, for the moment, a fixed volume form
$\epsilon_{bcd}$. In this case, we can specify a unique connection from the
projective class by insisting that $\nabla_a\epsilon_{bcd}=0$. In any case, as
in Section~\ref{attack}, we have a well-defined splitting $H=P\oplus V$.

To complete the geometric structure on $M$, we suppose that we are given two
linearly independent $1$-forms, say $\phi$ and $\psi$. These forms are
sufficient to define a $G_2$ contact structure on (an open subset of)
$C$ as follows.

For $P\hookrightarrow C$, define a frame $e^1,e^2\in\Gamma(P)$ by requiring
that
$$\begin{array}{cc}
e^1\intprod\pi^*\phi=1&\enskip e^1\intprod\pi^*\psi=0\\[4pt]
e^2\intprod\pi^*\phi=0&\enskip e^2\intprod\pi^*\psi=1.
\end{array}$$
This is legitimate wherever $\pi^*\phi$ and $\pi^*\psi$ are linearly 
independent when restricted to $H\subset TC$ (and defines an open subset 
of~$C$). Next, recall the canonical short exact sequence
$$0\to P\to\pi^*TM\to L\to 0$$
on $C$ and hence a canonical identification
$$L^*=\pi^*\Wedge_M^3\otimes\Wedge^2P.$$
Thus, we may use the nowhere vanishing sections
$\pi^*\epsilon_{bcd}\in\Gamma(\pi^*\Wedge_M^3)$ and $e^1\wedge
e^2\in\Gamma(\Wedge^2P)$ to trivialise $L^*$. With this trivialisation, the
Levi form $P\otimes V\to L$ identifies $V=P^*$ and we take
$e_1,e_2\in\Gamma(V)$ to be the dual frame to $e^1,e^2\in\Gamma(P)$. Finally,
we define a twisted cubic 
${\mathbb{RP}}_1\hookrightarrow{\mathbb{P}}(P\oplus V)={\mathbb{P}}(H)$ by
\begin{equation}\label{twisted_cubic}{\mathbb{RP}}_1\ni[s,t]\longmapsto
[s^3\,e^1+s^2t\,e^2+t^3\,e_1-3st^2\,e_2]
\in{\mathbb{P}}(H).\end{equation}
The seemingly peculiar choice of constants here is to ensure 
that this cubic induce the existing Levi form on~$H$, specifically that
$$\begin{array}{l}(s\tilde{t}-t\tilde{s})^3\\
=(s^3\,e^1+s^2t\,e^2)\intprod(\tilde{t^3}\,e_1-3\tilde{s}\tilde{t}^2\,e_2)
-(\tilde{s}^3\,e^1+\tilde{s}^2\tilde{t}\,e^2)\intprod(t^3\,e_1-3st^2\,e_2)
\end{array}$$
in accordance with~(\ref{canonical_symplectic_form}).

There is some `gauge freedom' associated with this arrangement,
i.e.~changes in the data $(\epsilon,\phi,\psi)$ on $M$ that do not affect
the associated $G_2$ projective structure on~$C$. Specifically, if we replace
$(\epsilon_{bcd},\phi_b,\psi_b)$ by
\begin{equation}\label{freedom}\hat\epsilon_{bcd}=\Omega^4\epsilon_{bcd}\qquad
\hat\phi_b=h^3\phi_b\qquad\hat\psi_b=\Omega\psi_b,\end{equation}
for arbitrary smooth nowhere vanishing functions $\Omega$ and $h$ on $M$, then
$$\hat e^1=h^{-3}e^1\qquad\hat e^2=\Omega^{-1}e^2\qquad
\hat e_1=h^6\Omega^{-3}e_1\qquad\hat e_2=h^3\Omega^{-2}e_2,$$
giving rise to the twisted cubic
$$[\hat s,\hat t]\longmapsto
[h^{-3}\hat s^3\,e^1+\Omega^{-1}\hat s^2\hat t\,e^2
+h^6\Omega^{-3}\hat t^3\,e_1-3h^3\Omega^{-2}\hat s\hat t^2\,e_2],$$
which is simply a reparameterisation, namely 
$[\hat s,\hat t]=[hs,h^{-2}\Omega t]$, of the original
cubic~(\ref{twisted_cubic}). Using the language of {\em projective weights\/},
e.g.~as in~\cite{DE}, we may regard $\epsilon_{bcd}$ as the tautologically
defined section of $\Wedge^3(4)$ and the true data as a
$1$-form~$\phi\in\Gamma(\Wedge^1)$, defined only up to scale, and
$\psi\in\Gamma(\Wedge^1(1))$, a $1$-form of projective weight~$1$.

The construction above is almost captured geometrically as follows. Firstly,
the projective structure on $M$ splits the contact distribution as 
$H=P\oplus V$, which, in the bundle ${\mathbb{P}}(H)\to M$ of 
$3$-dimensional projective spaces, may be viewed as a
family of skew lines
$$
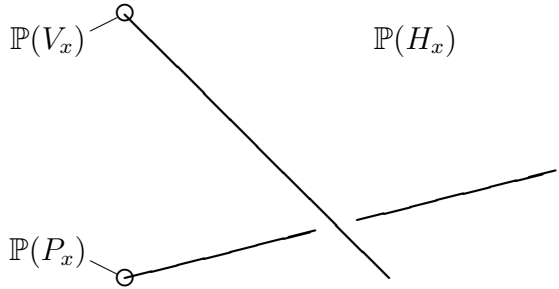
\begin{picture}(100,100)
\put(110,90){\makebox(0,0){${\mathbb{P}}(H_x)$}}
\put(-29,90){\makebox(0,0){${\mathbb{P}}(V_x)$}}
\put(-29,10){\makebox(0,0){${\mathbb{P}}(P_x)$}}
\put(-3,98.5){\line(-2,-1){10}}
\put(-3,1.5){\line(-2,1){10}}
\thicklines
\put(0,100){\line(1,-1){100}}
\put(0,0){\line(4,1){72}}
\put(88,22){\line(4,1){75}}
\put(0,100){\makebox(0,0){\Large$\circ$}}
\put(0,0){\makebox(0,0){\Large$\circ$}}
\end{picture}$$
and the family of twisted cubics~(\ref{twisted_cubic}) looks like this:
$$
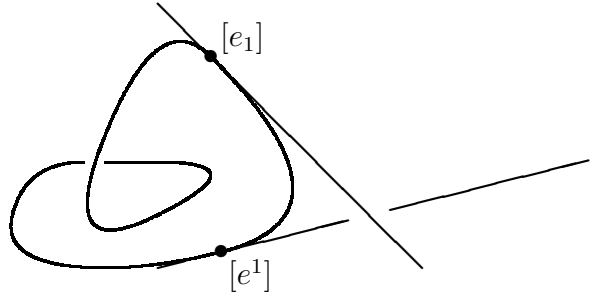
\begin{picture}(100,110)(0,-10)\thicklines
\put(0,100){\line(1,-1){100}}
\put(0,0){\line(4,1){72}}
\put(88,22){\line(4,1){75}}
\put(20,80){\makebox(0,0){$\bullet$}}
\put(24,6){\makebox(0,0){$\bullet$}}
\qbezier(20,80)(80,20)(24,6)
\qbezier(20,80)(0,100)(-20,50)
\qbezier(-20,50)(-40,0)(0,20)
\qbezier(0,20)(40,40)(0,40)
\qbezier(0,40)(-20,40)(-20,40)
\qbezier(-28,40)(-50,40)(-55,20)
\qbezier(-55,20)(-60,0)(-20,0)
\qbezier(-20,0)(0,0)(24,6)
\put(32,87){\makebox(0,0){$[e_1]$}}
\put(35,-3){\makebox(0,0){$[e^1]$}}
\end{picture}$$
intersecting ${\mathbb{P}}(P)$ and ${\mathbb{P}}(V)$ tangentially at $[e^1]$
and $[e_1]$, respectively, as can be seen from~(\ref{twisted_cubic}). Requiring
that the twisted cubic be compatible with the Levi form, as we do, is 
insufficient to fix it. The gauge freedom explained above says that there is 
just one more scalar-valued piece of information required at each point 
$x\in M$ and this may be interpreted as our requiring
$\psi\in\Gamma(\Wedge^1(1))$.

As is detailed in~\cite{EN}, there are two examples of this construction for
which $C$ turns out to be (locally) the flat model \raisebox{1pt}{$\gtwo{}{}$}
for the split real form of the Lie group~$G_2$. For both, we start with the
standard flat projective structure on ${\mathbb{R}}^3$ with usual 
co\"ordinates $(x,y,z)$ and take
$$\epsilon_{abc}=dx\wedge dy\wedge dz,$$
the standard volume form. 

\medskip{\bf First example.}\quad We take
$$\phi=dx\quad\mbox{and}\quad\psi=dy.$$
With slightly different co\"ordinates, this example is due to Engel~\cite{E}. 
In~\cite{EN} we write down all its $14$-dimensional symmetries.

\medskip{\bf Second example.}\quad We take
$$\phi=x\,dy-y\,dx\quad\mbox{and}\quad\psi=y^{-1}dy.$$
The geometry on $M$ appears to be different from the first example but we shall
see in Theorem~\ref{G2_flat_thm} below that the induced $G_2$ contact structure
on $C$ is $G_2$-flat and therefore isomorphic.

\section{The geometry on $M$}\label{geometry_on_M}
In this section we speculate on the geometry on $M$ that is needed to generate
the $G_2$ contact structure on $C\to M$, as in~\S\ref{G2}. We have already seen
in (\ref{freedom}) that $\psi$ should have projective weight~$1$ whilst $\phi$
may be arbitrarily rescaled. In fact, we shall suppose that $\phi$ has
projective weight~$2$. There are several reasons for this, the most na\"{\i}ve
of which is as follows. For a $1$-form of projective weight $w$, the formula
(\ref{projective_change}) for projective change becomes
\begin{equation}\label{projective_change_with_weight}\widehat\nabla_a\phi_b
=\nabla_a\phi_b+(w-1)\Upsilon_a\phi_b-\Upsilon_b\phi_a\end{equation}
and when $w=2$, we see that $\nabla_{(a}\phi_{b)}$ is invariant, where the
round brackets mean to take the symmetric part. We may, therefore, assume that
$\nabla_{(a}\phi_{b)}=0$ as a sort of {\em compatibility\/} between $\phi_b$
and the projective structure defined by~$\nabla_a$. We shall come back to this
shortly but an immediate and congenial consequence of imposing
$\nabla_{(a}\phi_{b)}=0$ is that $\phi_b$ is then determined up to an overall 
constant:
$$\nabla_{(a}(f\phi_{b)})=f\nabla_{(a}\phi_{b)}+(\nabla_{(a}f)\phi_{b)}
\enskip\Rightarrow\enskip\nabla_af=0.$$

In summary, the data we are supposing on $M$ is as follows.
\begin{equation}\label{data}\begin{array}{ll}
\bullet&\mbox{A projective structure, determined by~$\nabla_a$},\\
\bullet&\mbox{$\phi_b\in\Gamma(M,\Wedge^1(2))$, such that 
$\nabla_{(a}\phi_{b)}=0$},\\
\bullet&\mbox{$\psi_b\in\Gamma(M,\Wedge^1(1))$},
\end{array}\end{equation}
and we note that the two examples from \S\ref{G2} above satisfy
$\nabla_{(a}\phi_{b)}=0$, as requested. There are various projective invariants
that we may generate from this data. The {\em concircularity\/} operator
$$\textstyle\theta^b\mapsto(\nabla_a\theta^b)_\circ
\equiv\nabla_a\theta^b-\frac13\delta_a{}^b\nabla_c\theta^c$$
(where $\circ$ means to take the trace-free part) is projectively invariant if
$\theta^a$ has projective weight~$-1$. Meanwhile, the tautological form
$\epsilon^{bcd}$ has projective weight~$-4$ so $\epsilon^{bcd}\phi_c\psi_d$ has
projective weight~$-1$ and hence
$(\nabla_a(\epsilon^{bcd}\phi_c\psi_d))_\circ$ is invariant. It may be 
rewritten as
$$\nabla_a(\phi_{[b}\psi_{c]})-\nabla_{[a}(\phi_b\psi_{c]}).$$
It is easily verified that this expression vanishes for the two examples given
at the end of the previous section. When $w=1$ the projective change
(\ref{projective_change_with_weight}) reads
$$\widehat\nabla_b\psi_c=\nabla_b\psi_c-\Upsilon_c\psi_b\qquad\mbox{whence}
\qquad\psi_{[a}\nabla_{b]}\psi_c\enskip\mbox{is invariant}$$
and also vanishes for our two examples.

\begin{thm}\label{G2_flat_thm} For the data \eqref{data} to define a
flat $G_2$ contact structure on~$C$, it is necessary and sufficient 
that
\begin{equation}\label{G2_flat_conditions}
\nabla_a\mbox{ be projectively flat},\enskip 
\nabla_a(\phi_{[b}\psi_{c]})=\nabla_{[a}(\phi_b\psi_{c]}),\enskip
\psi_{[a}\nabla_{b]}\psi_c=0.\end{equation}
\end{thm}

The proof will be given shortly but, firstly, some discussion. This theorem
leads us to the following.

\medskip{\bf Third example.}\quad We take the standard flat projective 
structure on ${\mathbb{R}}^3$ with usual co\"ordinates $(x,y,z)$ and
$$\phi=x\,dy-y\,dx\quad\mbox{and}\quad\psi=zy^{-1}dy-dz.$$
It is easily verified that all conditions
(\ref{G2_flat_conditions}) hold and the corresponding $G_2$
contact structure on the configuration space~$C$ is, therefore, flat.

\medskip{\bf Remarks on these three examples.}\quad Although these examples
seem na\"{i}vely to be distinct (and from the Riemannian viewpoint, this is
true), they are, in fact, projectively equivalent. Specifically, if we set
$$\hat x=-x/y\qquad\hat y=-1/y\qquad\hat z=-z/y$$
then
$$d\hat x=y^{-2}(x\,dy-y\,dx)\quad\mbox{and}\quad d\hat y=y^{-2}dy$$
and, taking into account that $\phi$ should have projective weight $2$ and
$\psi$ should have projective weight $1$, it follows that the first example,
$(\phi=d\hat x, \psi=d\hat y)$ is converted into 
$(\phi=x\,dy-y\,dx,\psi=y^{-1}\,dy)$, which is the second example. Similarly, 
the projective change
$$\hat x=-x/y\qquad\hat y=-z/y\qquad\hat z=1/y$$
converts $(\phi=d\hat x, \psi=d\hat y)$ into
$(\phi=x\,dy-y\,dx,\psi=zy^{-1}\,dy-dz)$, which is the third example.

\smallskip To some extent, this projective equivalence of our three examples
justifies our request that the $1$-form $\phi$ should have projective weight
$2$ but there is another good reason for this, as follows. Recall,
in~(\ref{freedom}), that $\phi$ may be arbitrarily rescaled without effecting
the resulting $G_2$ contact geometry on~$C$. Its kernel $D\subset TM$ is
therefore canonically defined and we may write
\begin{equation}\label{SES}
0\to D\to TM\xrightarrow{\,\phi\,}\xi\to 0,\end{equation} a short exact
sequence tautologically defining a line bundle $\xi$ on~$M$. Let us temporarily
suppose that $D$ is a contact distribution (even though this is false in our
three examples). In three dimensions, the Levi form then provides a
canonical isomorphism $\Wedge^2D=\xi$ and, feeding this back into~(\ref{SES}),
we may identify $\Wedge^3TM=\xi^2$. That $M$ is oriented allows us to identify
$\xi$ as the bundle of densities of projective weight~$2$ (whether or not $M$
has a projective structure). In summary, when $D$ is a contact distribution we
are forced to regard $\phi\in\Gamma(M,\Wedge^0(2))$ and, even when $D$ is
integrable, we may {\em choose\/} to do this. As already observed, in the
presence of a projective structure~$[\nabla_a]$, we may also insist that
$\nabla_{(a}\phi_{b)}=0$. In case $D$ is a contact distribution, this is
exactly the compatibility required between $D$ and a projective structure in
order that the pair $([\nabla_a],D)$ define a {\em contact projective
structure\/} in the sense of Harrison~\cite{H} and/or Fox~\cite{F}. It is a
type of parabolic geometry, the flat model of which is ${\mathbb{RP}}_3$ under
the action of ${\mathrm{Sp}}(4,{\mathbb{R}})$. 

It would be nice to construct a flat $G_2$ contact structure starting with this
flat contact projective structure. Unfortunately, this seems to be impossible.
Specifically, in standard co\"ordinates $(x,y,z)$ on~${\mathbb{R}}^3$, we
may arrange that
$$\phi=x\,dy-y\,dx+dz.$$
On the other hand, the general solution of the concircularity equation 
$(\nabla_a\theta^b)_\circ=0$ is
$$\theta
=a\frac\partial{\partial x}+b\frac\partial{\partial y}
+c\frac\partial{\partial z}
+e\Big(x\frac\partial{\partial x}+y\frac\partial{\partial y}
+z\frac\partial{\partial z}\Big),$$
in which case
$$\theta^b\phi_b=c+bx-ay+ez.$$
Therefore, we cannot find $\psi_d\not=0$ such that
$(\nabla_a(\epsilon^{bcd}\phi_c\psi_d))_\circ=0$, which is the second condition
from (\ref{G2_flat_conditions}). As we shall see in the proof of
Theorem~\ref{construct} below, the quantity $\epsilon^{bcd}\phi_c\psi_d$
naturally arises in constructing the $G_2$ contact geometry on $C$ and, even if
Theorem~\ref{G2_flat_thm} is too restrictive, one would expect the projective
invariant $(\nabla_a(\epsilon^{bcd}\phi_c\psi_d))_\circ$ to be part of the
harmonic curvature of the $G_2$ contact geometry on $C$ (which, in general, is
a binary septic: in fact, a section of $\gtwo{7}{-4}$\,).

Taking $\phi\in\Gamma(M,\Wedge^1(2))$ and $\psi\in\Gamma(M,\Wedge^1(1))$,
allows us to write the twisted cubic cone (\ref{twisted_cubic}) inside the
contact distribution $H\subset TC$ more invariantly than was done
in~\S\ref{G2}. We obtain the following. 

\begin{thm}\label{construct}
Suppose $M$ is a three-dimensional smooth manifold and write
$C\xrightarrow{\,\pi\,}M$ for the configuration space of flying saucers in~$M$.
Given data in the form~\eqref{data} on $M$, we may canonically construct a
$G_2$ contact structure on a suitable open subset of $C$ so that the contact
distribution $H\subset TM$ is written as
$$\textstyle H=\bigodot^3\!S,\quad\mbox{where }
S=\pi^*\Wedge^0(2/3)\oplus\pi^*\Wedge^0(-1/3).$$
\end{thm}
\begin{proof}
Recall from (\ref{dynkin}) that
$$P=\xox{1}{1}{-1}\quad\mbox{and}\quad V=\xox{-1}{1}{1}.$$
Also, from (\ref{dynkin}) we have
$$0\to\xox{1}{1}{-1}\to\pi^*TM\to\xox{1}{0}{1}\to 0,$$
a canonical short exact sequence on $C$ and, in particular, canonical 
surjections
$$\pi^*\Wedge_M^1\to\big(\,\xox{1}{1}{-1}\;\big)^*=\xox{-2}{1}{0}
\quad\mbox{and}\quad\pi^*\Wedge_M^1(w)\to\enskip\xox{w-2\;{}}{1}{0}$$
for any projective weight~$w$. Writing $\pi^!$ for the pullback $\pi^*$
followed by this surjection, firstly gives
$$\Theta\equiv\pi^!\phi\wedge\pi^!\psi
\in\Gamma(C,\xox{0}{1}{0}\wedge\xox{-1}{1}{0})=\Gamma(C,\xox{0}{0}{1})$$
and then, on the open set where $\Theta$ is non-vanishing,
$$E_1\equiv\Theta\,\pi^!\phi\in\Gamma(C,\xox{0}{1}{1})\qquad
E_2\equiv\Theta\,\pi^!\psi\in\Gamma(C,\xox{-1}{1}{1})$$
and 
$$E^1\equiv-\Theta^{-1}\pi^!\psi\in\Gamma(C,\xox{-1}{1}{-1})\qquad
E^2\equiv\Theta^{-1}\pi^!\phi\in\Gamma(C,\xox{0}{1}{-1}).$$
Bearing in mind that
$$\begin{array}{ll}
\xox{0}{1}{1}=V\otimes\Wedge_C^0(1)&\qquad
\xox{-1}{1}{1}=V\\[4pt]
\xox{-1}{1}{-1}=P\otimes\Wedge_C^0(-2)&\qquad
\xox{0}{1}{-1}=P\otimes\Wedge_C^0(-1)\,,
\end{array}$$
where $\Wedge_C^0(w)=\xox{w}{0}{0}=\pi^*\Wedge_M^0(w)$, we conclude that
$$S\equiv\pi^*\Wedge^0(2/3)\oplus\pi^*\Wedge^0(-1/3)
\ni\hspace{-77pt}\begin{array}[t]{cl}(\sigma,\tau)\\
\begin{picture}(10,15)
\put(5,5){\makebox(0,0){$\downarrow$}}
\put(5,10.2){\makebox(0,0){${}-{}$}}\end{picture}\\
\sigma^3\,E^1+\sigma^2\tau\,E^2+\tau^3\,E_1-3\sigma\tau^2\,E_2&
{}\hspace{-8pt}\in P\oplus V=H
\end{array}$$
is well defined. It is an invariant formulation of (\ref{twisted_cubic}) whose
range defines a twisted cubic cone in ${\mathbb{P}}(H)$ compatible with the
Levi form and hence a $G_2$ contact structure on~$C$.
\end{proof}
We remark that, although the projective weights $2/3$ and $-1/3$ in the
identification of $S$ may look contrived, they yield
$$\Wedge^2S\otimes\Wedge^2S\otimes\Wedge^2S
=\pi^*\Wedge^0(1)$$
and therefore, in accordance with the vector bundle version of~(\ref{Wedge2H}),
that $\pi^*\Wedge^0(w)=\gtwo{0}{w}$, as one might expect.

\medskip\noindent{\em Proof of Theorem~\ref{G2_flat_thm}.}\enskip Since
$\nabla_a$ is projectively flat we may suppose, without loss of generality,
that our manifold is ${\mathbb{R}}^3\hookrightarrow{\mathbb{RP}}_3$ with
$\nabla_a$ the standard flat connection. The operator
$\phi_b\mapsto\nabla_{(a}\phi_{b)}$ is the first BGG operator
$$\xoo{0}{1}{0}\xrightarrow{\,\nabla\,}\xoo{-2}{2}{0}\quad\mbox{on 
${\mathbb{RP}}_3$}$$
with kernel the irreducible representation
$\ooo{0}{1}{0}=\Wedge^2{\mathbb{R}}^4$ of ${\mathrm{SL}}(4,{\mathbb{R}})$
(acting by projective transformations on~${\mathbb{RP}}_3$). There are two
non-zero orbits for the action of ${\mathrm{SL}}(4,{\mathbb{R}})$ on
$\Wedge^4{\mathbb{R}}^4$ depending on rank and the non-degenerate case is
represented by $\phi=x\,dy-y\,dx+dz$, which we have already seen to be
incompatible with the second condition of~(\ref{G2_flat_conditions}). It 
follows that, without loss of generality, we may suppose~$\phi=dx$. 

The operator $\theta^b\mapsto(\nabla_a\theta^b)_\circ$ is also a first BGG 
operator 
$$\xoo{0}{0}{1}\xrightarrow{\,\nabla\,}\xoo{-2}{1}{1}$$
whose solution space is ${\mathbb{R}}^4=\ooo{0}{0}{1}$ as an
${\mathrm{SL}}(4,{\mathbb{R}})$-module. The degenerate $2$-form corresponding
to $\phi$ specifies a $2$-plane in ${\mathbb{R}}^4$ and so there are just two
cases for $\theta^b$ depending on whether the corresponding vector in 
${\mathbb{R}}^4$ lies in this plane or not. This is exactly whether 
$\theta^a\phi_a$ vanishes or not and, with $\theta^a$ being of the form 
$\epsilon^{abc}\phi_b\psi_c$, it must vanish. Therefore, without loss of 
generality $\theta^a=\partial/\partial z$. We have reached the following 
normal forms for $\phi$ and~$\psi$:
$$\phi=dx\quad\mbox{and}\quad\psi=\xi(x,y,z)\,dx+dy$$
where $\xi(x,y,z)$ is an arbitrary smooth function. It remains to consider the
remaining condition from~(\ref{G2_flat_conditions}), namely that
$\psi_{[a}\nabla_{b]}\psi_c=0$. It means that $\xi=\xi(x,y)$, a function of
$(x,y)$ alone, and that $\xi\xi_y=\xi_x$. In the computation that follows, we
shall find $\xi\xi_y-\xi_x$ as the only non-trivial component of the harmonic
curvature for the associated $G_2$ contact structure and our proof will be
complete.

The harmonic curvature is computed in Theorem~\ref{harmonic_curvature} below.
To use it we must specify the $G_2$ contact structure on $C$ in terms of an
adapted co-frame~(\ref{fully_adapted}). Starting with 
$$\omega^1=\phi=dx\quad\mbox{and}\quad\omega^2=\psi=\xi(x,y)\,dx+dy$$
on $M={\mathbb{R}}^3$ we may take 
$$\textstyle\omega^0=dz-a\,dx-b\,dy\qquad\omega^3=-\frac13\,db\qquad
\omega^4=da-\xi(x,y)\,db$$
in local co\"ordinates $(x,y,z,a,b)$ on~$C$, as in~\S\ref{contact}. These 
satisfy (\ref{fully_adapted}) with $\chi\equiv 1$. We compute
$$\begin{array}{ll}d\omega^1=0\quad&
d\omega^2=d\xi\wedge dx=-\xi_y\,\omega^1\wedge\omega^2\\[4pt]
d\omega^3=0&d\omega^4=-d\xi\wedge db
=3(\xi_x-\xi\xi_y)\,\omega^1\wedge\omega^3+3\xi_y\,\omega^2\wedge\omega^3
\end{array}$$
and we see that the only non-zero coefficients in (\ref{we_write}) are
$$a^2{}_{12}=-\xi_y\qquad a^4{}_{13}=3(\xi_x-\xi\xi_y)\qquad 
a^4{}_{23}=3\xi_y.$$
Substituting into (\ref{psi}) gives $\psi_0=\psi_1=\psi_2=\psi_3=\psi_4
=\psi_5=\psi_7=0$ and $\psi_6=6(\xi\xi_y-\xi_x)$, as claimed.
\hfill$\square$

\medskip{\bf Remarks.}\quad The partial differential equation $\xi\xi_y=\xi_x$ 
has plenty of local solutions. Indeed, if $F(t)$ is an arbitrary smooth 
function and we define $\xi(x,y)$ implicitly by the equation
$$F(\xi)=x\xi+y,$$
then 
$\xi\xi_y=\xi_x$. In particular, the foliation of ${\mathbb{R}}^3$ defined by
$\psi=\xi\,dx+dy$ need not be the planar foliation exhibited in the three
examples above. Therefore, we have found many projectively inequivalent
examples of structures given by data of the form (\ref{data}) that all give
rise to the same (flat) $G_2$ contact structure on the associated configuration
space.

\section*{Appendix: harmonic curvature of a $G_2$ contact structure}
As already mentioned (see ~\cite{CS} for the general theory), the harmonic
curvature of a $G_2$ contact structure is a section of the
bundle~$\gtwo{7}{-4}$. This binary septic may be obtained by the Cartan
equivalence method.

Specifically, a $G_2$ contact structure on a five-dimensional manifold $C$ may
be specified in terms of an adapted co-frame as follows. Firstly, we choose a
$1$-form $\omega^0$ whose kernel is the contact distribution $H\subset TC$. 
Non-degeneracy of the contact distribution says that 
$\omega^0\wedge d\omega^0\wedge d\omega^0\not=0$. The $G_2$ structure is then 
determined by completing $\omega^0$ to a co-frame 
$\omega^0,\omega^1,\omega^2,\omega^3,\omega^4$ so that
\begin{equation}\label{adapted}
d\omega^0=\chi(\omega^1\wedge\omega^4-3\omega^2\wedge\omega^3)
\enskip\bmod\omega^0\end{equation}
for some smooth function~$\chi$. Specifically, if $X_0,X_1,X_2,X_3,X_4$ is the 
dual frame, then $H={\mathrm{span}}\{X_1,X_2,X_3,X_4\}$ and the twisted cubic 
(\ref{twisted_cubic}) may be given as
\begin{equation}\label{twisted_cubic_again}
(s,t)\mapsto s^3X_1+s^2tX_2+st^2X_3+t^3X_4,\end{equation}
compatibility with the Levi form being a consequence of~(\ref{adapted}). 
Imposing the structure equations (\ref{adapted}) leaves precisely the 
following freedom in choice of co-frame:
\begin{equation}\label{thisisG}
\left[\!\!\begin{array}c\tilde\omega^0\\ \tilde\omega^1\\ 
\tilde\omega^2\\ \tilde\omega^3\\ \tilde\omega^4\end{array}\!\!\right]
\!\!=\!\!\left[\!\!\addtolength{\arraycolsep}{-1pt}\begin{array}{ccccc}
t_9&0&0&0&0\\ 
t_{10}&t_5{}^3&3t_5{}^2t_6&3t_5t_6{}^2&t_6{}^3\\
t_{11}&t_5{}^2t_7&t_5(t_5t_8+2t_6t_7)&t_6(2t_5t_8+t_6t_7)&t_6{}^2t_8\\ 
t_{12}&t_5t_7{}^2&t_7(2t_5t_8+t_6t_7)&t_8(t_5t_8+2t_6t_7)&t_6t_8{}^2\\ 
t_{13}&t_7{}^3&3t_7{}^2t_8&3t_7t_8{}^2&t_8{}^3
\end{array}\!\!\right]\!\!
\left[\!\!\begin{array}c\omega^0\\ \omega^1\\ 
\omega^2\\ \omega^3\\ \omega^4\end{array}\!\!\right]\end{equation}
for arbitrary functions $t_5,t_6,t_7,t_8,t_9,t_{10},t_{11},t_{12},t_{13}$ on
$C$ subject only to $t_9(t_5t_8-t_6t_7)\not=0$. (The functions
$t_5,t_6,t_7,t_8$ correspond to 
$$\raisebox{4pt}{$(s,t)\mapsto(s,t)$}
\left(\!\begin{array}{cc}t_5&t_7\\ t_6&t_8\end{array}\right)$$
as a change of parameterisation in (\ref{twisted_cubic_again}) whilst
$t_9,t_{10},t_{11},t_{12},t_{13}$ modify the co-frame with multiples
of~$\omega^0$.) To proceed with Cartan's method of equivalence, we now pass to
the bundle $\widetilde{C}\to C$ whose sections are frames adapted according
to~(\ref{adapted}). It is a $G$-principal bundle where $G$ is the
$9$-dimensional Lie subgroup of ${\mathrm{GL}}(5,{\mathbb{R}})$ given
in~(\ref{thisisG}) and comes tautologically equipped with $1$-forms
$\theta^0,\theta^1,\theta^2,\theta^3,\theta^4$ whose pull-backs along a section
are $\omega^0,\omega^1,\omega^2,\omega^3,\omega^4$, the co-frame on $C$
corresponding to that section. Cartan's aim is to extend this to an invariant
co-frame on $\widetilde{C}$ by making various normalisations. For our purposes
we need not take these normalisations too far. For calculation, we choose a
co-frame on $C$ adapted according to (\ref{adapted}) so that $\widetilde{C}$ is
identified as $G\times C$ and the forms
$\theta^0,\theta^1\,\theta^2,\theta^3,\theta^4$ are given as
$$\left[\!\!\begin{array}c\theta^0\\ \theta^1\\ 
\theta^2\\ \theta^3\\ \theta^4\end{array}\!\!\right]
\!\!=\!\!\left[\!\!\addtolength{\arraycolsep}{-1pt}\begin{array}{ccccc}
t_9&0&0&0&0\\ 
t_{10}&t_5{}^3&3t_5{}^2t_6&3t_5t_6{}^2&t_6{}^3\\
t_{11}&t_5{}^2t_7&t_5(t_5t_8+2t_6t_7)&t_6(2t_5t_8+t_6t_7)&t_6{}^2t_8\\ 
t_{12}&t_5t_7{}^2&t_7(2t_5t_8+t_6t_7)&t_8(t_5t_8+2t_6t_7)&t_6t_8{}^2\\ 
t_{13}&t_7{}^3&3t_7{}^2t_8&3t_7t_8{}^2&t_8{}^3
\end{array}\!\!\right]\!\!
\left[\!\!\begin{array}c\omega^0\\ \omega^1\\ 
\omega^2\\ \omega^3\\ \omega^4\end{array}\!\!\right].$$

\smallskip{\bf Step 0}\enskip Normalise the co-frame on $C$ so that
\begin{equation}\label{fully_adapted}
d\omega^0=\chi(\omega^1\wedge\omega^4-3\omega^2\wedge\omega^3).\end{equation}
This is easily achieved by the freedom
$$\tilde\omega^1=\omega^1+t_{10}\omega^0,\enskip
\tilde\omega^2=\omega^1+t_{11}\omega^0,\enskip
\tilde\omega^3=\omega^1+t_{12}\omega^0,\enskip
\tilde\omega^4=\omega^1+t_{13}\omega^0$$
and determines the functions $t_{10},t_{11},t_{12},t_{13}$.

\smallskip{\bf Step 1}\enskip Find $\theta^5$ such that
$$d\theta^0
=-6\theta^0\wedge\theta^5+\theta^1\wedge\theta^4-3\theta^2\wedge\theta^3.$$
This may be achieved by setting $-t_9=\Delta^3/\chi$, where 
$\Delta=t_6t_7-t_5t_8$ and then
$$\theta^5=-\frac16\frac{d\chi}{\chi}+\frac12\frac{d\Delta}{\Delta}
+\frac{\chi}{\Delta^3}\Big(\frac{t_{13}\theta^1}6-\frac{t_{12}\theta^2}2
+\frac{t_{11}\theta^3}2-\frac{t_{10}\theta^4}6\Big)+s\theta^0$$
for some function~$s$.

\smallskip{\bf Step 2}\enskip Find $\theta^7,\theta^8,\theta^9$ such that 
$$E^1\equiv d\theta^1-(6\theta^0\wedge\theta^9-3\theta^1\wedge\theta^5
-3\theta^1\wedge\theta^8+3\theta^2\wedge\theta^7)$$
is of the form $c^1{}_{\mu\nu}\theta^\mu\wedge\theta^v$ for 
$\mu,\nu=0,1,2,3,4$. It follows that
$$E^1\wedge\theta^0\wedge\theta^1\wedge\theta^2
=c^1{}_{34}\theta^0\wedge\theta^1\wedge\theta^2\wedge\theta^4\wedge\theta^5.$$
It turns out that
$$c^1{}_{34}=\frac{t_5{}^7}{\Delta^5}
\Big(\psi_0+\psi_1s+\psi_2s^2+\psi_3s^3+\psi_4s^4+\psi_5s^5+\psi_6s^6
+\psi_7s^7\Big),$$
where $\psi_0,\psi_1,\psi_2,\psi_3,\psi_4,\psi_5,\psi_6,\psi_7$ are functions 
on~$C$. These are the coefficients of the invariantly defined harmonic 
curvature. In fact, if
$$d\omega^0=\chi(\omega^1\wedge\omega^4-3\omega^2\wedge\omega^3)$$
on $C$ and we write
\begin{equation}\label{we_write}
d\omega^i=\sum_{1\leq j<k\leq4}a^i{}_{jk}\omega^j\wedge\omega^k
\enskip\bmod\omega^0,\enskip\mbox{for }i=1,2,3,4\end{equation}
then 
\begin{equation}\label{psi}\begin{array}{rcl}\psi_0&=&a^1{}_{34}\\
\psi_1&=&-2a^1{}_{24}+3a^2{}_{34}\\
\psi_2&=&3a^1{}_{14}+a^1{}_{23}-6a^2{}_{24}+3a{}^3{}_{34}\\
\psi_3&=&-2a^1{}_{13}+9a^2{}_{14}+3a^2{}_{23}-6a^3{}_{24}+a^4{}_{34}\\
\psi_4&=&a^1{}_{12}-6a^2{}_{13}+9a^3{}_{14}+3a^3{}_{23}-2a^4{}_{24}\\
\psi_5&=&3a^2{}_{12}-6a^3{}_{13}+3a^4{}_{14}+a^4{}_{23}\\
\psi_6&=&3a^3{}_{12}-2a^4{}_{13}\\
\psi_7&=&a^4{}_{12}.
\end{array}\end{equation}
{From} the general theory of parabolic geometry~\cite[Theorem~3.1.12]{CS} we find
the following characterisation of flat $G_2$ contact structures. 
\begin{thm}\label{harmonic_curvature}
The local symmetry algebra of the $G_2$ contact structure specified by a
co-frame adapted according to \eqref{fully_adapted} is the split exceptional
Lie algebra $G_2$ if and only if
$\psi_0,\psi_1,\psi_2,\psi_3,\psi_4,\psi_5,\psi_6,\psi_7$ given by \eqref{psi}
all vanish. Moreover, in this case, the manifold $C$ is locally isomorphic to
the homogeneous model.
\end{thm}

\medskip{\bf Remarks.}\quad The formul{\ae} (\ref{we_write}) and (\ref{psi})
may alternatively be derived as follows. According to~(\ref{extension}), the 
exterior derivative $d:\Wedge^1\to\Wedge^2$ gives rise, via the diagram
$$\begin{array}{ccccccccc}
0&\to&\gtwo{0}{-1}&\to&\Wedge^1&\to&\gtwo{3}{-2}&\to&0\\
&&&&d\downarrow\phantom{d}\\
0&\to&\gtwo{3}{-3}&\to&\Wedge^2&\to&
\gtwo{0}{-1}\oplus\gtwo{4}{-3}&\to&0,\\
\end{array}$$
to an invariantly defined first order differential operator 
\begin{equation}\label{rumin}\nabla:\gtwo{3}{-2}\to\gtwo{4}{-3}.
\end{equation}
In fact, this is nothing more than the second operator in the Rumin 
complex~\cite{R}, which depends only on the contact structure on~$C$. The 
exterior derivative, on the other hand, may be written as 
$\omega_b\mapsto\nabla_{[a}\omega_{b]}$ for any {\em torsion-free\/} 
connection $\nabla_a$ on $TC$. Thus, the Rumin operator 
(\ref{rumin}) may be written with {\em spinor indices\/}~\cite{OT}, adapted to 
our cause, as
\begin{equation}\label{rumin_with_spinors}
\omega_{ABC}\longmapsto\nabla_{(AB}{}^H\omega_{CD)H}.\end{equation}
One may readily check that the formul{\ae} (\ref{we_write}) and (\ref{psi}) 
amount to the stipulation that
\begin{equation}\label{stip}
\psi_{ABCDEFG}\pi^A\pi^B\pi^C\pi^D\pi^E\pi^F\pi^G
=\pi^A\pi^B\pi^C\pi^D\nabla_{AB}{}^H(\pi_C\pi_D\pi_H)\end{equation}
for all sections $\pi_A$ of $S^*=\gtwo{1}{\enskip-2/3}$\enskip. Note, by the 
Leibniz rule
$$\begin{array}{rcl}\pi^A\pi^B\pi^C\pi^D\nabla_{AB}{}^H(f\pi_C\pi_D\pi_H)
&=&f\pi^A\pi^B\pi^C\pi^D\nabla_{AB}{}^H(\pi_C\pi_D\pi_H)\\
&&\enskip{}+\pi^A\pi^B\underbrace{\pi^C\pi^D\pi_C}_{=0}\pi_D\pi_H
\nabla_{AB}{}^Hf,\end{array}$$
that the right hand side of (\ref{stip}) is homogeneous of degree $7$ over the
functions and, therefore, automatically of the form given on the left. It 
follows that $\psi_{ABCDEFG}$ is the obstruction to writing (\ref{rumin}) as
$$\omega_{ABC}\longmapsto{\mathcal{D}}_{(AB}{}^H\omega_{CD)H},$$
where ${\mathcal{D}}_{ABC}$ is induced by a connection on 
$S^*=\gtwo{1}{\enskip-2/3}$\quad, because the Leibniz rule for such a 
connection would imply that
$${\mathcal{D}}_{AB}{}^H(\pi_C\pi_D\pi_H)
=\pi_C\pi_D{\mathcal{D}}_{AB}{}^H\pi_H
+\pi_C\pi_H{\mathcal{D}}_{AB}{}^H\pi_D
+\pi_D\pi_H{\mathcal{D}}_{AB}{}^H\pi_C$$
and the right hand side of (\ref{stip}) would therefore vanish. In other words,
the formula (\ref{rumin_with_spinors}) depends on $\nabla_a$ being torsion-free
and $\psi_{ABCDEFG}$ may be seen as some invariant part of the 
{\em partial torsion\/} of a freely chosen spinor connection
${\mathcal{D}}_a:S\to\Wedge^1\otimes S$.

More specifically, suppose ${\mathcal{D}}_a:S\to\Wedge^1\otimes S$ is any
connection and define its {\em partial torsion\/} 
$T_{ABCD}{}^{EFG}=T_{(ABCD)}{}^{(EFG)}$ according to
$${\mathcal{D}}_{(AB}{}^E{\mathcal{D}}_{CD)E}f
=T_{ABCD}{}^{EFG}{\mathcal{D}}_{EFG}f,\enskip
\forall\mbox{ smooth functions }f.$$
Changing the connection~${\mathcal{D}}_a$, leads to a change of 
{\em partial connection\/} ${\mathcal{D}}_{ABC}:S^*\to\gtwo{3}{-2}\otimes S^*$
according to
$$\widehat{\mathcal{D}}_{ABC}\pi_D
={\mathcal{D}}_{ABC}\pi_D-\Gamma_{ABCD}{}^E\pi_E,
\enskip\mbox{where }\Gamma_{ABCD}{}^E=\Gamma_{(ABC)D}{}^E$$
and, therefore, an induced change on 
$\gtwo{3}{-2}=\bigodot^3\!S^*$, namely 
$$\widehat{\mathcal{D}}_{ABC}\omega_{DEF}=
{\mathcal{D}}_{ABC}\omega_{DEF}-3\Gamma_{ABC(D}{}^G\omega_{EF)G}.$$  
It follows that
$$\widehat{\mathcal{D}}_{AB}{}^E\widehat{\mathcal{D}}_{CDE}f
={\mathcal{D}}_{AB}{}^E{\mathcal{D}}_{CDE}f
-3\Gamma_{AB}{}^E{}_{(C}{}^G{\mathcal{D}}_{DE)G}f$$
and, therefore, that
$$\widehat{\mathcal{D}}_{(AB}{}^E\widehat{\mathcal{D}}_{CD)E}f
\!=\!{\mathcal{D}}_{(AB}{}^E{\mathcal{D}}_{CD)E}f
-2\Gamma_{(AB}{}^E{}_C{}^F{\mathcal{D}}_{D)EF}f
+\Gamma_{H(AB}{}^{HE}{\mathcal{D}}_{CD)E}f$$
whence the partial torsion of ${\mathcal{D}}_a$ changes according to
$$\widehat{T}_{ABCD}{}^{EFG}=T_{ABCD}{}^{EFG}
-2\Gamma_{(AB}{}^{(E}{}_C{}^F\delta_{D)}{}^{G)}
+\Gamma_{H(AB}{}^{H(E}\delta_C{}^F\delta_{D)}{}^{G)}.$$
In particular, the trace-free part of $T_{ABCD}{}^{EFG}$, equivalently
$$\psi_{ABCDEFG}\equiv T_{(ABCDEFG)},$$
is an invariant of the $G_2$ contact structure.

\end{document}